\documentclass[a4paper]{amsart}
\usepackage{amssymb, enumitem}
\usepackage[all]{xy}
\usepackage{hyperref, aliascnt}
\usepackage{mathrsfs}

\def\today{\number\day\space\ifcase\month\or   January\or February\or
   March\or April\or May\or June\or   July\or August\or September\or
   October\or November\or December\fi\   \number\year}
\newtheorem{lma}{Lemma}[section]

\newaliascnt{thmCt}{lma}
\newtheorem{thm}[thmCt]{Theorem}
\aliascntresetthe{thmCt}

\newaliascnt{propCt}{lma}
\newtheorem{prop}[propCt]{Proposition}
\aliascntresetthe{propCt}

\newtheorem*{thm*}{Theorem}
\newtheorem*{cor*}{Corollary}
\newtheorem*{prop*}{Proposition}

\theoremstyle{definition}

\newaliascnt{pgrCt}{lma}

\aliascntresetthe{pgrCt}

\newaliascnt{rmkCt}{lma}
\newtheorem{rmk}[rmkCt]{Remark}
\aliascntresetthe{rmkCt}

\newaliascnt{pbmCt}{lma}
\newtheorem{pbm}[pbmCt]{Problem}
\aliascntresetthe{pbmCt}

\newaliascnt{ntnCt}{lma}
\newtheorem{ntn}[ntnCt]{Notation}
\aliascntresetthe{ntnCt}

\newcommand{\Z}{{\mathbb{Z}}}
\newcommand{\R}{{\mathbb{R}}}
\newcommand{\C}{{\mathbb{C}}}
\newcommand{\N}{{\mathbb{N}}}
\newcommand{\Hi}{{\mathcal{H}}}
\newcommand{\B}{{\mathcal{B}}}

\newcommand{\spec}{{\mathrm{sp}}}
\newcommand{\supp}{{\mathrm{supp}}}
\newcommand{\Max}{{\mathrm{Max}}}
\newcommand{\QSL}{Q\!S\!L}
\newcommand{\SL}{S\!L}

\newcommand{\lcg}{locally compact group}
\newcommand{\I}{\infty}

\title{Quotients of Banach algebras acting on $L^p$-spaces}
\date{\today}

\author[Eusebio Gardella]{Eusebio Gardella}
\address{Eusebio Gardella
Mathematisches Institut, Fachbereich Mathematik und Informatik der
Universit\"at M\"unster, Einsteinstrasse 62, 48149 M\"unster, Germany.}
\email{gardella@uni-muenster.de}

\author{Hannes Thiel}
\address{Hannes Thiel
Mathematisches Institut, Fachbereich Mathematik und Informatik der
Universit\"at M\"unster, Einsteinstrasse 62, 48149 M\"unster, Germany.}
\email{hannes.thiel@uni-muenster.de}
\urladdr{http://wwwmath.uni-muenster.de/u/hannes.thiel/}

\thanks{The first named author was partially supported by the D.~K. Harrison Prize from the
University of Oregon.
The second named author was partially supported by the Deutsche Forschungsgemeinschaft (SFB 878).}

\begin{document}

%==========================================================================================
\begin{abstract}
We show that the class of Banach algebras that can be isometrically represented on an $L^p$-space, for $p\neq 2$, is not closed under quotients.
This answers a question asked by Le Merdy 20 years ago.
Our methods are heavily reliant on our earlier study of Banach algebras generated by invertible isometries of $L^p$-spaces.
\end{abstract}

\maketitle

%==========================================================================================
\section{Introduction}

An operator algebra is a closed subalgebra of the algebra $\B(\Hi)$ of bounded linear operators on a Hilbert space $\Hi$.
If $A$ is an operator algebra and $I\subseteq A$ is a closed, two-sided ideal, then the quotient Banach
algebra $A/I$ is again an operator algebra, that is, it can be isometrically represented on a Hilbert space.
This classical result is due to Lumer and Bernard, although the commutative case (when $A$ is a uniform algebra) was proved earlier by Cole.

In Problem~3.8 of \cite{Mer96QuotSubalgBX}, Christian Le Merdy raised the question of whether this result can be generalized to Banach algebras acting on $L^p$-spaces for $p\in [1,\infty)$.
More precisely, if $\mathscr{E}$ is a class of Banach spaces, we say that a Banach algebra $A$ is an \emph{$\mathscr{E}$-operator algebra} if there exist a Banach space $E$ in $\mathscr{E}$ and an isometric homomorphism $\varphi\colon A\to \B(E)$.
Given $p\in [1,\infty)$, we consider the class $L^p$ of $L^p$-spaces, the class $\SL^p$ of Banach spaces that are isometrically isomorphic to subspaces of $L^p$-spaces, and the class $\QSL^p$ of Banach spaces that are quotients of $\SL^p$-spaces.

The Bernard-Cole-Lumer Theorem asserts that $L^2$-operator algebras are closed under quotients.
In Corollary~3.2 of \cite{Mer96QuotSubalgBX}, Le Merdy showed that $\QSL^p$-operator algebras are closed under quotients.
In Corollary~1.5.2.3 of \cite{Jun96Habilitationsschrift}, Marius Junge showed the analogous result for $\SL^p$-operator algebras.
Since the classes $\QSL^2$ and $\SL^2$ both agree with $L^2$, the results of Le Merdy and Junge are generalizations of the Bernard-Cole-Lumer Theorem.

As the authors point out in \cite{Mer96QuotSubalgBX} and \cite{Jun96Habilitationsschrift}, the arguments used there are not suitable to deal with the more natural class of $L^p$-operator algebras.
Indeed, the question of whether $L^p$-operator algebras are closed under quotients, for $p\neq 2$, remained open for 20 years.
The case $p=1$ of this question has recently been answered negatively in Theorem~6.2 of~\cite{GarThi14arX:LpGenInvIsom},
using a classical result of Malliavin on the failure of spectral synthesis for $\ell^1(\Z)$.
In this paper, we answer negatively the remaining cases of the question. (Even for $p=1$, we construct new examples of quotients
of $\ell^1(\Z)$ which cannot be represented on an $L^1$-space.
These are, in particular, semisimple, unlike those constructed in \cite{GarThi14arX:LpGenInvIsom}.)

For $p\in[1,\infty)$, we consider the algebra $F^p(\Z)$ of $p$-pseudofunctions on $\Z$.
This algebra was introduced in the early 70's by Herz in \cite{Her73SynthSubgps}, who denoted it
$PF_p(\Z)$.
Algebras of $p$-pseudofunctions (also for \lcg{s} other than $\Z$) have been studied in a number of places:
\cite{NeuRun09ColumnRowQSL}; \cite{Phi13arX:LpCrProd}; \cite{GarThi14arX:LpGenInvIsom}; \cite{GarThi14arX:LpGpAlg}; \cite{GarThi14arX:FunctLpGpAlg}, just to list a few.

The algebra $F^p(\Z)$ is a semisimple, commutative Banach algebra with spectrum $S^1$.
Given an open set $V$ in $S^1$, we let $I_V$ denote the largest closed two-sided ideal of $F^p(\Z)$ that is supported on $V$.
For $p\in [1,\infty)\setminus\{2\}$, and $V$ neither empty nor dense in $S^1$, we show that $F^p(\Z)/I_V$ is not an $L^p$-operator algebra;
see \autoref{thm:main}.
In fact, we even show that there is no injective, contractive homomorphism with closed range from $F^p(\Z)/I_V$ to the algebra of bounded linear operators on \emph{any} $L^q$-space for $q\in [1,\I)$;
see \autoref{rmk:mainExtended}.

Given the recent attention received by $L^p$-operator algebras, deciding whether these are closed
under quotients becomes more relevant and technically useful.
For example, consider the $L^p$-analogs $\mathcal{O}_n^p$ of the Cuntz algebras;
see \cite{Phi13arX:LpAnalogsCtz}.
These are all simple, and any contractive, non-zero representation of any of them
on an $L^p$-space is automatically injective (in fact, isometric).
For $p=2$, these two properties are well-known to be equivalent.
However, for $p\neq 2$, they are not, since quotients of $L^p$-operator algebras are not
in general representable on $L^p$-spaces.
These two properties of $\mathcal{O}_n^p$ therefore require separate and independent proofs.
A similar problem arises with the $L^p$-analogs of irrational rotation algebras;
see \cite{GarThi15pre:IrrRotAlgLp}.

%==========================================================================================
%==========================================================================================
\section{The example}

%==========================================================================================
A \emph{representation} of a Banach algebra $A$ on a Banach space $E$ is a contractive homomorphism $\varphi\colon A\to\B(E)$.
The representation is called \emph{unital} (\emph{injective}, \emph{isometric}) if $\varphi$ is.

We begin with some preparatory results.
Our first lemma allows us to assume that a unital $L^p$-operator algebra has a unital, isometric representations on some $L^p$-space.

%==========================================================================================
\begin{lma}
\label{lma:Corner}
Let $A$ be a unital Banach algebra and let $p\in [1,\I)$.
If $A$ can be (injectively, isometrically) represented on an $L^p$-space, then there is a \emph{unital} (injective, isometric) representation of $A$ on an $L^p$-space.
\end{lma}
\begin{proof}
Let $F$ be an $L^p$-space and let $\psi\colon A\to \B(F)$ be a contractive homomorphism.
Then $e=\psi(1)$ is an idempotent with $\|e\|=1$ in $\B(F)$.
By Theorem~6 in \cite{Tza69ContrProjLp}, the range $E$ of $e$ is an $L^p$-space.
The cut-down homomorphism
\[
\varphi\colon A\to \B(E)\cong \psi(1)\B(F)\psi(1)
\]
is the desired map.
It is clear that $\varphi$ is injective or isomeric whenever $\psi$ is, respectively.
\end{proof}

%==========================================================================================
If $A$ is a commutative unital Banach algebra, we will denote by $\Gamma_A\colon A\to C(\Max(A))$ its Gelfand transform, which is natural in the following sense.
If $\varphi\colon A\to B$ is a unital homomorphism between commutative unital Banach algebras $A$ and $B$, then
the assignment $\Max(B)\to \Max(A)$ given by $I\mapsto \varphi^{-1}(I)$ defines a contractive homomorphism
$\Gamma(\varphi)\colon C(\Max(A))\to C(\Max(B))$ making the following diagram commute:
\[
\xymatrix@R-5pt{
A\ar[rr]^-{\Gamma_A}\ar[d]_-\varphi && C(\Max(A))\ar[d]^-{\Gamma(\varphi)} \\
B\ar[rr]_-{\Gamma_B}&& C(\Max(B)).
}
\]

%==========================================================================================
\begin{rmk}
\label{rem:SpectraS1}
We recall the following fact about spectra of elements in Banach
algebras. If $B$ is a unital Banach algebra, $A$ is a subalgebra containing the unit of $B$,
and $a$ is an element of $A$ such that $\spec_A(a)=\partial \spec_A(a)$ as susbets of $\C$,
then $\spec_A(a)=\spec_B(a)$. This can be deduced, for example, from Theorem~ in~\cite{Pal94Book};
see also the inclusions in the middle of page 219 there.

In particular, if the spectrum of an element of a Banach algebra is a subset of $S^1$, 
then it is equal to the spectrum computed in any unital algebra containing the given algebra.
\end{rmk}

%==========================================================================================
Let $p\in[1,\infty)$.
We first recall some facts about $F^p(\Z)$;
the reader is referred to \cite{Phi13arX:LpCrProd}, \cite{GarThi14arX:LpGpAlg}, and \cite{GarThi14arX:FunctLpGpAlg} for more general versions of these statements, as well as their proofs.

The Banach algebra $F^p(\Z)$ is a subalgebra of $\B(\ell^p(\Z))$ that can be defined as follows.
Denote by $\lambda_p\colon \ell^1(\Z)\to \B(\ell^p(\Z))$ the left regular representation of $\Z$, which is
given by $\lambda_p(f)\xi=f\ast\xi$ for $f\in \ell^1(\Z)$ and $\xi\in \ell^p(\Z)$.
Then
\[
F^p(\Z)=\overline{\lambda_p(\ell^1(\Z))}\subseteq\B(\ell^p(\Z)).
\]
The algebra $F^p(\Z)$ also admits a description as a universal algebra with respect to invertible isometries on $L^p$-spaces.
The algebra $F^p(\Z)$ is clearly commutative, and its maximal ideal space is canonically homeomorphic to $S^1$.
Moreover, its Gelfand transform is injective. For $p=1$, there is a natural isometric identification
$F^1(\Z)=\ell^1(\Z)$,
while for $p$ and $q$ in $[1,\I)$, there is an abstract isometric isomorphism $F^p(\Z)\cong F^q(\Z)$ if and
only if $p$ and $q$ are either equal or (H\"older) conjugate;
see Theorem~3.4 in~\cite{GarThi14arX:FunctLpGpAlg}.

%==========================================================================================
\begin{ntn}
\label{ntn:A_IV}
Let $p\in[1,\infty)$.
We let $\Gamma$ denote the Gelfand transform from $F^p(\Z)$ to $C(S^1)$.
For an open subset $V$ of $S^1$, we set
\[
I_V=\Gamma^{-1}(C_0(V)),
\]
which is a closed, two-sided ideal in $F^p(\Z)$.
We will abbreviate $F^p(\Z)$ to $A$, and the quotient $F^p(\Z)/I_V$ to
$A_V$.
The Gelfand transform $\Gamma_{A_V}\colon A_V\to C(\Max(A_V))$ will be abbreviated to $\Gamma_V$.
\end{ntn}

%==========================================================================================
\begin{prop}
\label{prop:FirstStep}
Let $p\in[1,\infty)$, and let $V$ be a nonempty open subset of $S^1$.
We let $A=F^p(\Z)$, $I_V$ and $A_V$ be as in \autoref{ntn:A_IV}.
Suppose that there exists $q\in [1,\I)$ such that $A_V$ is an $L^q$-operator algebra.
Then the Gelfand transform $\Gamma_V\colon A_V\to C(S^1\setminus V)$ is an isomorphism (although not necessarily isometric).
In particular, and identifying $F^p(\Z)$ with a subalgebra of $C(S^1)$ via $\Gamma$, it follows that every continuous function on $S^1\setminus V$ is the restriction of a function in $F^p(\Z)$.
\end{prop}
\begin{proof}
It is clear that $\Max(A_V)$ is canonically homeomorphic to $S^1\setminus V$, so the range of $\Gamma_V$
can be canonically identified with a subalgebra of $C(S^1\setminus V)$.
We claim that $A_V$ is semisimple.
Denote by $\pi\colon A\to A_V$ the canonical quotient map, and
denote by $r\colon C(S^1)\to C(S^1\setminus V)$
the restriction map.
Observe that naturality of the Gelfand transform implies
that the following diagram is commutative:
\begin{equation}\label{diagram}
\xymatrix@R-5pt{
A\ar[rr]^-{\Gamma}\ar[d]_-\pi && C(S^1)\ar[d]^-{r}\\
A_V\ar[rr]_-{\Gamma_V}&& C(S^1\setminus V)
.}
\end{equation}
It is a standard fact that semisimplicity of $A_V$ is equivalent to injectivity of $\Gamma_V$. Let
$a\in A_V$ satisfy $\Gamma_V(a)=0$. Find $b\in A$ such that $\pi(b)=a$. Then $r(\Gamma(b))=0$, so
$\Gamma(b)$ belongs to $\ker(r)=C_0(V)$. In other words, $b\in \Gamma^{-1}(C_0(V))=I_V$. Hence $a=0$
and $A_V$ is semisimple.

It follows that there are natural identifications

\[ A_V\cong \frac{\Gamma(F^p(\Z))}{\Gamma(F^p(\Z))\cap C_0(V)}\cong \frac{\Gamma(F^p(\Z))+C_0(V)}{C_0(V)}.\]

Observe that $A_V$ is generated by the image $\pi(u)$ of the canonical generator $u$ of $A=F^p(\Z)$, which is an invertible isometry.
Suppose that there exist $q\in [1,\I)$, an $L^q$-space $E$, and an isometric representation $\varphi\colon A_V\to \B(E)$. By \autoref{lma:Corner}, we can assume that $\varphi$ is unital.
It is clear that $\varphi(\pi(u))$ generates $\varphi(A_V)$. Since $\varphi$ is unital, $\varphi(\pi(u))$ is an invertible
isometry of an $L^q$-space. Moreover, using \autoref{rem:SpectraS1} at the first step, we have
\[\spec_{\B(E)}(\varphi(\pi(u)))=\spec_{\varphi(A_V)}(\varphi(\pi(u)))=\spec_{A_V}(\pi(u))=S^1\setminus V.\]

We claim that the Gelfand transform $\Gamma_{\varphi(A_V)}\colon\varphi(A_V)\to C(S^1\setminus V)$ is an isomorphism.
Once we show this, it will follow that $\Gamma_V$ is also an isomorphism.

First, $\Gamma_{\varphi(A_V)}$ is clearly injective by semisimplicity of $A_V$.
Suppose that $q=2$. Then $\varphi(A_V)$ is a $C^*$-algebra, because it is generated by an invertible
isometry of a Hilbert space (a unitary), and $A_V$ is therefore self-adjoint. The claim is then an immediate consequence of Gelfand's theorem (and in this case $\Gamma_{\varphi(A_V)}$ is isometric).
Assume now that $q\in [1,\I)\setminus\{2\}$. In this case, and since the spectrum of $\varphi(\pi(u))$ in $\B(E)$ is not
the whole circle, the result follows from part~(1) of Corollary~5.21 in~\cite{GarThi14arX:LpGenInvIsom}.
The claim is proved, and the first part of the proposition follows.

For the second assertion in the statement of the proposition, consider the commutative diagram
in~(\ref{diagram}). It follows that for every $f\in C(S^1\setminus V)$,
there exists $g\in A=F^p(\Z)$ such that $\Gamma_V(\pi(g))=f$. Regarding $g$ as a function on $S^1$, this is
equivalent to $g|_{S^1\setminus V}=f$.
\end{proof}

%==========================================================================================
Let $\theta\in\R$.
Then it is easy to show that the homeomorphism $h_\theta\colon S^1\to S^1$ given by
$h_\theta(\zeta)=e^{2\pi i \theta}\zeta$ for $\zeta\in S^1$ induces an isometric automorphism of $F^p(\Z)$. (We warn the reader that it is not in general true that any homeomorphism of $S^1$ induces an isometric, or even contractive, automorphism of $F^p(\Z)$.
In fact, when $p\neq 2$, the only homeomorphisms of $S^1$ that do so are the rotations and compositions of rotations with the homeomorphism $\zeta\mapsto\overline{\zeta}$ of $S^1$.)

The following is the main result of this paper.
It asserts that for every $p\neq 2$, there is a quotient algebra of $F^p(\Z)$ that has no isometric representation on any $L^q$-space for $q\in[1,\infty)$.

%==========================================================================================
\begin{thm}
\label{thm:main}
Let $p\in [1,\I)\setminus\{2\}$, let $V$ be a nonempty open subset of $S^1$.
We let $A=F^p(\Z)$, $I_V$ and $A_V$ be as in \autoref{ntn:A_IV}.
Assume that $V$ is not dense in $S^1$.
Then $A_V$ cannot be isometrically represented on \emph{any} $L^q$-space for $q\in [1,\I)$.
\end{thm}
\begin{proof}
We argue by contradiction, so let $V$ be an open subset of $S^1$ as in the statement, and suppose that there exists
$q\in [1,\I)$ such that $A_V$ can be isometrically represented on an $L^q$-space.

Let $f\in C(S^1)$. We claim that $f$ belongs to $\Gamma(F^p(\Z))$. Once we prove this, it will follow from part~(2)
of Corollary~3.20 in \cite{GarThi14arX:LpGpAlg} that $p=2$, and hence the proof will be complete.

Let $W$ be a nonempty open subset of $S^1$ such that $V\cap W=\emptyset$. With the notation used in the comments before this
theorem, and using compactness of $S^1$, find $n\in\N$ and $\theta_1,\ldots,\theta_n\in\R$ such that
$\bigcup\limits_{j=1}^n h_{\theta_j}(W)=S^1$.
For $j\in \{1,\ldots,n\}$, set $V_j=h_{\theta_j}(V)$ and $W_j=h_{\theta_j}(W)$. There is an isometric isomorphism
\[A_{h_{\theta_j}(V)}\cong A_V,\]
so the Banach algebra $A_{h_{\theta_j}(V)}$ can be isomorphically represented on an $L^q$-space. It
follows from \autoref{prop:FirstStep} that every continuous function on $S^1\setminus V_j$ is in the image
of $\Gamma_{V_j}$. In particular, every continuous function on $\overline{W_j}$ is in the image of $\Gamma_{V_j}$.

From now on, we identify the algebras $A, A_{V_1},\ldots,A_{V_n}$ with their images under their Gelfand transforms.
In particular, for $j=1,\ldots,n$, every continuous function on $\overline{W_j}$ is the restriction of a function in $A$.

Choose continuous functions $k_1,\ldots,k_n\colon S^1\to \R$ satisfying
\begin{enumerate}
\item $0\leq k_j\leq 1$ for $j=1,\ldots,n$;
\item $\supp(k_j)\subseteq W_j$ for $j=1,\ldots,n$;
\item $\sum\limits_{j=1}^n k_j(\zeta)=1$ for all $\zeta\in S^1$;
\item $k_j$ belongs to $F^p(\Z)$ for $j=1,\ldots,n$ (for example, take $k_j\in C^\I(S^1)$).
\end{enumerate}

For $j=1,\ldots,n$, choose a function $g_j\in F^p(\Z)$ such that $(g_j)|_{W_j}=f|_{W_j}$.
Then the product $g_jk_j$ belongs to $F^p(\Z)$ because each of the factors does.
Since the support of $k_j$ is contained in $W_j$, and $f$ and $g_j$ agree on $W_j$, we have
$fk_j=g_jk_j$ for $j=1,\ldots,n$. Now,
\[f=f\cdot\left(\sum_{j=1}^nk_j \right)=\sum_{j=1}^n g_jk_j,\]
so $f$ belongs to $F^p(\Z)$, and the claim is proved.

We have shown that Gelfand transform $\Gamma\colon F^p(\Z)\to C(S^1)$ is surjective.
Since $F^2(\Z)$ is canonically isomorphic to $C(S^1)$, we must have $p=2$ by part~(2) of Corollary~3.20 in \cite{GarThi14arX:LpGpAlg}. This is a contradiction, and the result follows.
\end{proof}

%==========================================================================================
\begin{rmk}
\label{rmk:mainExtended}
The same argument as in \autoref{prop:FirstStep} and \autoref{thm:main} shows that if $V\subseteq S^1$
is a non-trivial, non-dense open subset, then there is no contractive, injective
representation of $A_V$ on any $L^q$-space \emph{with closed range}.
On the other hand, there certainly are contractive, injective representations of $A_V$ on $L^q$-spaces for every $q\in [1,\I)$.
Indeed, for $q\in [1,\I)$, any commutative semisimple Banach algebra $B$ can be represented on an $L^q$-space (although rarely can it be represented with closed range).
Indeed, $C_0(\Max(B))$ is isometrically represented on an $L^q$-space for every $q$, and such a representation can be
composed with the Gelfand transform to get the desired contractive representation of $B$ on an $L^q$-space.
\end{rmk}

%==========================================================================================
\begin{rmk}
In contrast to \autoref{thm:main}, some (non-trivial) quotients of $F^p(\Z)$ are isometrically
representable on $L^p$-spaces.
For example, if $V$ is the complement of the set of $n$-th roots of unity in $S^1$ for some $n\in\N$, then $F^p(\Z)/I_V$ is canonically isometrically isomorphic to $F^p(\Z_n)$.
(This identification is induced by the quotient map $\Z\to\Z_n$.)
An analogous statement holds for the translates of $V$.
We do not know, however, whether these are the only quotients of $F^p(\Z)$ that can be represented on $L^p$-spaces.
We therefore suggest:
\end{rmk}

%==========================================================================================
\begin{pbm}
\label{pbm:Characterize}
Characterize those ideals $I$ of $F^p(\Z)$ such that $F^p(\Z)/I$ can be isometrically represented on an $L^p$-space.
\end{pbm}

%==========================================================================================
We do not know whether $F^p(\Z)$ has spectral synthesis, except for $p=1$ (in which case it does
not) and $p=2$ (in which case it does). Since Banach algebras generated by an invertible isometry
of an $L^p$-space together with its inverse are automatically semisimple by the results in
\cite{GarThi14arX:LpGenInvIsom}, we conclude that for $F^p(\Z)/I$ to be isometrically representable
on an $L^p$-space, there must exist an open subset $V\subseteq S^1$ such that $I=I_V$ (and $V$ must
be dense by \autoref{thm:main}). This means that \autoref{pbm:Characterize} can be solved without
knowing whether $F^p(\Z)$ has spectral synthesis, that is, without knowing whether every ideal of
$F^p(\Z)$ is of the form $I_V$.

We do not know whether density of $V$ is sufficient for $F^p(\Z)/I_V$ to be representable on an $L^p$-space.

We conclude this paper with an observation. If $A$ is a Banach algebra and $a\in A$, we denote by
$B(a)$ the smallest Banach subalgebra of $A$ containing $a$.

%==========================================================================================
\begin{rmk}
Suppose that $p$ is not an even integer, and let $V\subseteq S^1$ be a nonempty, non-dense open subset.
We let $A=F^p(\Z)$, $I_V$ and $A_V$ be as in \autoref{ntn:A_IV}.
It follows from Corollary~1.5.2.3 in \cite{Jun96Habilitationsschrift} that $A_V$ is an $\SL^p$-operator algebra, so there exists an invertible isometry $v$ of an $\SL^p$-space $E$ such that $B(u)$ is isometrically isomorphic to $A_V$.
By Theorem~I in \cite{Rud76LpIsomEquim}, there exist a $L^p$-space $F$ containing $E$ as a closed subspace, and a canonical invertible isometry $w$ of $F$ extending $v$.
By naturality of the construction, one may be tempted to guess that $B(v)$ and $B(w)$ are isometrically isomorphic.
However, \autoref{thm:main} shows that this is not the case.
\end{rmk}

%==========================================================================================
%==========================================================================================
%\bibliographystyle{aomalphaMy}
%\bibliography{ReferencesMR}

\providecommand{\bysame}{\leavevmode\hbox to3em{\hrulefill}\thinspace}
\providecommand{\noopsort}[1]{}
\providecommand{\mr}[1]{\href{http://www.ams.org/mathscinet-getitem?mr=#1}{MR~#1}}
\providecommand{\zbl}[1]{\href{http://www.zentralblatt-math.org/zmath/en/search/?q=an:#1}{Zbl~#1}}
\providecommand{\jfm}[1]{\href{http://www.emis.de/cgi-bin/JFM-item?#1}{JFM~#1}}
\providecommand{\arxiv}[1]{\href{http://www.arxiv.org/abs/#1}{arXiv~#1}}
\providecommand{\MR}{\relax\ifhmode\unskip\space\fi MR }
\providecommand{\MRhref}[2]{%
  \href{http://www.ams.org/mathscinet-getitem?mr=#1}{#2}
}
\providecommand{\href}[2]{#2}

\end{document}